\documentclass[runningheads]{llncs}

\usepackage{amssymb, amsmath, mathrsfs}
\usepackage{synttree, bussproofs,scrextend, stmaryrd, rotating, xcolor, algorithm2e, upgreek}
\usepackage{esvect,pgf, tikz, color}
\usetikzlibrary{arrows, automata}
\usepackage[all]{xy}
\usepackage{changepage,enumerate,proof,upgreek,relsize,trfsigns,
comment}

\newcommand{\D}{\mathrm{D}_{\charx}}
\newcommand{\T}{\mathrm{T}_{\charx}}
\newcommand{\four}{\mathrm{4}_{\charx}}
\newcommand{\B}{\mathrm{B}_{\charx}}
\newcommand{\five}{\mathrm{5}_{\charx}}
\newcommand{\ipa}{\mathrm{IPA}}
  
\newcommand{\canmod}{M^{C}}
\newcommand{\candom}{W^{C}}
\newcommand{\canirel}{\leq^{C}}
\newcommand{\canmrel}{R_{\charx}^{C}}
\newcommand{\canv}{V^{C}}

\newcommand{\ifandonlyif}{\textit{iff} }
\newcommand{\etc}{$\ldots$ }
\newcommand{\dfn}{Def.}
\newcommand{\fig}{Fig.}
\newcommand{\lem}{Lem.}
\newcommand{\thm}{Thm.}

\newcommand{\sect}{Sect.}
\newcommand{\cptr}{Ch.}

\newcommand{\h}{\mathsf{H}}
            
\newcommand{\ikm}{\mathsf{IK_{m}}}
\newcommand{\ikma}{\mathsf{IK_{m}}(\albet,\mathcal{A})}

\newcommand{\der}{\vdash_{\axs}^{\albet}}
\newcommand{\ent}{\Vdash_{\axs}^{\albet}}
\newcommand{\sat}{\Vdash^{\albet}}

\newcommand{\charx}{x}
\newcommand{\chary}{y}
\newcommand{\charz}{z}
\newcommand{\chara}{a}
\newcommand{\charb}{b}
\newcommand{\charc}{c}

\newcommand{\conv}[1]{\overline{#1}}

\newcommand{\iimp}{\supset}

\newcommand{\ieq}{\supset\subset}
\newcommand{\inot}{{\sim}}

\newcommand{\xbox}{[\charx]}
\newcommand{\xdia}{\langle \charx \rangle}
\newcommand{\xboxc}{[\conv{\charx}]}
\newcommand{\xdiac}{\langle \conv{\charx} \rangle}

\newcommand{\dia}{\Diamond}
\newcommand{\diap}[1]{\langle #1 \rangle}
\newcommand{\boxp}[1]{[ #1 ]}

\newcommand{\albet}{\Upsigma}
\newcommand{\albetf}{\albet^{+}}
\newcommand{\albetb}{\albet^{-}}
\newcommand{\lang}[1]{\mathcal{L}(#1)}

\newcommand{\prop}{\Upphi}

\newcommand{\axs}{\mathcal{A}}

\newcommand{\fseti}{\mathscr{A}}
\newcommand{\fsetii}{\mathscr{B}}
\newcommand{\fsetiii}{\mathscr{C}}
\newcommand{\fsetiv}{\mathscr{D}}

\newcommand{\nec}{(nec)}

\begin{document}

\title{A Framework for Intuitionistic Grammar Logics\thanks{Work supported by the European Research Council (ERC) Consolidator Grant 771779 (DeciGUT).}} 

\author{Tim S. Lyon\orcidID{0000-0003-3214-0828}}

\institute{Computational Logic Group, Institute of Artificial Intelligence, Technische Universit\"at Dresden, Germany  \\ \email{timothy\_stephen.lyon@tu-dresden.de}}


\maketitle

\begin{abstract}
We generalize intuitionistic tense logics to the multi-modal case by placing grammar logics on an intuitionistic footing. We provide axiomatizations for a class of base intuitionistic grammar logics as well as provide axiomatizations for extensions with combinations of seriality axioms and what we call \emph{intuitionistic path axioms}. We show that each axiomatization is sound and complete with completeness being shown via a typical canonical model construction.
\end{abstract}

\section{Introduction}\label{sec:introduction}



Having been introduced in 1988 by Fari\~nas del Cerro and Penttonen~\cite{CerPen88}, \emph{grammar logics} form a prominent class of normal, multi-modal logics that extend classical propositional logic with a set of modalities indexed by characters from an alphabet. Such logics obtain their name due to the incorporation of axioms which can be viewed as production rules in a context-free grammar, and which generate sequences of edges (which can be viewed as words) in a relational model. More significantly however, the class of grammar logics includes many well-known logics that have practical value in computer science; e.g. description logics~\cite{HorSat04}, epistemic logics~\cite{FagMosHalVar95}, information logics~\cite{Vak86}, temporal logics~\cite{CerHer95}, and standard modal logics (e.g. $\mathsf{K}$, $\mathsf{S4}$, and $\mathsf{S5}$)~\cite{DemNiv05}.


Another logical paradigm that is useful within computer science is that of \emph{constructive reasoning} (e.g.~\cite{How80,OsoNavArr04}), that is, reasoning where the claimed existence of an object implies its constructibility~\cite{Bro75}. One of the most renowned logics for formalizing constructive reasoning is \emph{intuitionistic logic}, which employs a version of implication that is stronger than its classical counterpart. Resting on the philosophical work of L.E.J. Brouwer, propositional intuitionistic logic was provided axiomatizations in the early 20\textsuperscript{th} century by Kolmogorov~\cite{Kol25}, Orlov~\cite{Orl28}, and Glivenko~\cite{Gli29}, with a first-order axiomatization given by Heyting~\cite{Hey30}. 


The interest in modal and intuitionistic logics naturally gave rise to combinations of the two, thus giving birth to the paradigm of \emph{intuitionistic modal logics}. A diverse set of intuitionistic modal logics have been proposed in the literature~\cite{BiePai00,BovDov84,Dov85,Fit48,PloSti86,Ser84,Sim94}, though the class of logics introduced by Plotkin and Stirling~\cite{PloSti86} 
 has become (most notably through the work of Simpson~\cite{Sim94}) one of the most popular 
 formulations. In the same year that Plotkin and Stirling~\cite{PloSti86} introduced their intuitionistic modal logics, Ewald introduced \emph{intuitionistic tense logic}~\cite{Ewa86}, which not only includes modalities that make reference to the future in a relational model ($\dia$ and $\Box$), but also includes modalities that make reference to the past ($\blacklozenge$ and $\blacksquare$). As with (multi-)modal and intuitionistic logics, intuitionistic modal logics have proven useful in computer science; e.g. such logics have been used to design verification techniques~\cite{FaiMen95}, in reasoning about functional programs~\cite{Pit91}, and in the definition of programming languages~\cite{DavPfe01}.


Due to the practical import of the aforementioned logics, it seems both natural and worthwhile to formulate intuitionistic versions of grammar logics. Hence, the main goal of this paper will be to axiomatize intuitionistic context-free grammar logics with converses, thus generalizing the work of~\cite{CerPen88,Ewa86,PloSti86}. 
 In the following section (\sect~\ref{sec:log-prelims}), we axiomatize and provide a semantics for intuitionistic grammar logics. Afterward (in \sect~\ref{sec:sound-complete}), we prove the soundness and completeness of our logics, with completeness being shown on the basis of a standard canonical model construction (adapting techniques provided in~\cite{Ewa86}). In the final section (\sect~\ref{sec:conclusion}), we briefly conclude and discuss future work.

\section{Intuitionistic Grammar Logics}\label{sec:log-prelims}


We define our languages for intuitionistic grammar logics relative to an \emph{alphabet} $\albet$ consisting of a non-empty countable set of characters, which will be used to index modalities. Following~\cite{DemNiv05}, we stipulate that each alphabet $\albet$ can be partitioned into a \emph{forward part} $\albetf := \{\chara, \charb, \charc, \ldots\}$ and a \emph{backward part} $\albetb := \{\conv{\chara}, \conv{\charb}, \conv{\charc}, \ldots\}$ such that each part has the same cardinality and the following is satisfied:
$$
\albet := \albetf \cup \albetb \text{ where } \albetf \cap \albetb = \emptyset \text{ and } \chara \in \albetf \text{ \ifandonlyif } \conv{\chara} \in \albetb
$$
We use $\chara$, $\charb$, $\charc$, \etc (possibly annotated) to denote the \emph{forward characters} contained in the forward part $\albetf$, and $\conv{\chara}$, $\conv{\charb}$, $\conv{\charc}$, \etc (possibly annotated) to denote the \emph{backward characters} contained in the backward part $\albetb$. Both forward and backward characters are referred to as \emph{characters} more generally, and we use $\charx$, $\chary$, $\charz$, \etc (possibly annotated) to denote them. Intuitively, modalities indexed with forward characters make reference to future states within a relational model, and modalities indexed with backward characters make reference to past states. The \emph{converse operation} $\conv{\cdot}$ is defined to be a function mapping each forward character $\chara \in \albetf$ to its \emph{converse} $\conv{\chara} \in \albetb$ and vice versa; hence, the converse operation is its own inverse, i.e. for any $\charx \in \albet$, $\charx = \conv{\conv{\charx}}$.

Each of our languages includes \emph{propositional atoms} from the denumerable set $\prop := \{p, q, r, \ldots\}$. Each language $\lang{\albet}$ is defined via the following grammar in BNF:
$$
A ::= p \ | \ \bot \ | \ A \lor A \ | \ A \land A \ | \ A \iimp A \ | \ \xdia A \ | \ \xbox A
$$
where $p$ ranges over the set of propositional atoms $\prop$ and $\charx$ ranges over the alphabet $\albet$. We use $A$, $B$, $C$, \etc to range over formulae in $\lang{\albet}$, define $\inot A := A \iimp \bot$, and define $A \ieq B := (A \iimp B)\land(B \iimp A)$. We interpret formulae on \emph{bi-relational $\albet$-models}, which are inspired by the models for intuitionistic modal and tense logics presented in~\cite{BovDov84,Dov85,Ewa86,PloSti86}: 

\begin{definition}[Bi-relational $\albet$-Model]\label{def:bi-relational-model} We define a \emph{bi-relational $\albet$-model} to be a tuple $M = (W, \leq, \{R_{\charx} \ | \ \charx \in \albet\}, V)$ such that:
\begin{itemize}

\item $W$ is a non-empty set of \emph{worlds} $\{w, u, v, \ldots\}$;

\item The \emph{intuitionistic relation} $\leq \ \subseteq W \times W$ is a preorder, i.e. it is reflexive and transitive;

\item The \emph{accessibility relation} $R_{\charx} \subseteq W \times W$ satisfies:

\begin{itemize}

\item[(F1)] For all $w, v, v' \in W$, if $w R_{\charx} v$ and $v \leq v'$, then there exists a $w' \in W$ such that $w \leq w'$ and $w' R_{\charx} v'$;

\item[(F2)] For all $w, w', v \in W$, if $w \leq w'$ and $w R_{\charx} v$, then there exists a $v' \in W$ such that $w' R_{\charx} v'$ and $v \leq v'$;

\item[(F3)] $w R_{\charx} u$ \ifandonlyif $u R_{\conv{\charx}} w$;

\end{itemize}


\item $V : W \to 2^{\prop}$ is a \emph{valuation function} satisfying the \emph{monotonicity condition}:  for each $w, u \in W$, if $w \leq u$, then $V(w) \subseteq V(u)$.

\end{itemize}

\end{definition}

The (F1) and (F2) conditions ensure the monotonicity of complex formulae (see \lem~\ref{lem:monotonicity}) in our models, which is a property characteristic of intuitionistic logics.\footnote{For a discussion of these conditions and their encompassing literature, see~\cite[\cptr~3]{Sim94}.} We note that we interpret accessibility relations indexed with forward characters as relating worlds to \emph{future} worlds, and accessibility relations indexed with backward characters as relating worlds to \emph{past} worlds. Such an interpretation shows that our models have a tense character, and additionally, shows that our logics generalize the intuitionistic tense logics of~\cite{Ewa86}.

We interpret formulae from $\lang{\albet}$ over bi-relational models via the following clauses.

\begin{definition}[Semantic Clauses]
\label{def:semantic-clauses} Let $M$ be a bi-relational $\albet$-model with $w \in W$. The \emph{satisfaction relation} $M,w \sat A$ between $w \in W$ of $M$ and a formula $A \in \lang{\albet}$ is inductively defined as follows:

\begin{itemize}

\item $M,w \sat p$ \ifandonlyif $p \in V(w)$, for $p \in \prop$;

\item $M,w \not\sat \bot$;

\item $M,w \sat A \lor B$ \ifandonlyif $M,w \sat A$ or $M,w \sat B$;

\item $M,w \sat A \land B$ \ifandonlyif $M,w \sat A$ and $M,w \sat B$;


\item $M,w \sat A \iimp B$ \ifandonlyif for all $w' \in W$, if $w \leq w'$ and $M,w' \sat A$, then $M,w' \sat B$;

\item $M,w \sat \xdia A$ \ifandonlyif there exists a $v \in W$ such that $w R_{\charx} v$ and $M,v \sat A$;

\item $M,w \sat \xbox A$ \ifandonlyif for all $w', v' \in W$, if $w \leq w'$ and $w' R_{\charx} v'$, then $M,v' \sat A$.

\end{itemize}
\end{definition}

\begin{lemma}\label{lem:monotonicity}
Let $M$ be a bi-relational $\albet$-model with $w,u \in W$ of $M$. If $w \leq u$ and $M, w \sat A$, then $M, u \sat A$.
\end{lemma}

\begin{proof} By induction on the complexity of $A$.
\qed
\end{proof}

As will be shown in the subsequent section, given an alphabet $\albet$, the set of formulae valid with respect to the class of bi-relational $\albet$-models is axiomatizable. We refer to the axiomatization as $\h\ikm(\albet)$ (with $\h$ denoting the fact that the axiomatization is a \emph{Hilbert calculus}), and call the corresponding logic that it generates $\ikm(\albet)$. We note that $\ikm(\albet)$ is taken to be the base intuitionistic grammar logic relative to $\albet$; below, we will also consider extensions of $\ikm(\albet)$ by extending its axiomatization with common modal axioms.

\begin{definition}[Axiomatization] 
 We define our axiomatization $\h\ikm(\albet)$ below, where we have an axiom and inference rule for each $\charx \in \albet$.

\begin{multicols}{2}
\begin{itemize}

\item[A0] Any axiomatization for intuitionistic propositional logic

\item[A1] $\xbox (A \iimp B) \iimp (\xbox A \iimp \xbox B)$

\item[A2] $\xbox (A \land B) \ieq (\xbox A \land \xbox B)$

\item[A3] $\xdia (A \lor B) \ieq (\xdia A \lor \xdia B)$

\item[A4] $\xbox (A \iimp B) \iimp (\xdia A \iimp \xdia B)$

\item[A5] $\xbox A \land \xdia B \iimp \xdia (A \land B)$

\item[A6] $\inot \xdia \bot$ 

\item[A7] $(A \iimp \xbox \xdiac A) \land (\xdia \xboxc A \iimp A)$

\item[A8] $(\xdia A \iimp \xbox B) \iimp \xbox (A \iimp B)$

\item[A9] $\xdia (A \iimp B) \iimp (\xbox A \iimp \xdia B)$

\item[R1] \AxiomC{$A$}\RightLabel{$\nec$}\UnaryInfC{$\xbox A$}\DisplayProof

\end{itemize}
\end{multicols}
We define the logic $\ikm(\albet)$ to be the smallest set of formulae from $\lang{\albet}$ closed under substitutions of the axioms and applications of the inference rules. A formula $A$ is defined to be a \emph{theorem} of $\ikm(\albet)$  \ifandonlyif $A  \in \ikm(\albet)$.
\end{definition}

We also consider logics that are extensions of $\ikm(\albet)$ with sets $\axs$ of the following axioms. 
$$
\D: \xbox A \iimp \xdia A \quad \ipa: (\diap{\charx_{1}} \cdots \diap{\charx_{n}} A \iimp \diap{\charx} A) \land (\boxp{\charx} A \iimp \boxp{\charx_{1}} \cdots \boxp{\charx_{n}} A)
$$
We refer to axioms of the form shown above left as \emph{seriality axioms}, and axioms of the form shown above right as \emph{intuitionistic path axioms} ($\ipa$s). We use $\axs$ to denote any arbitrary collection of axioms of the above forms. Moreover, we note that the collection of $\ipa$s includes multi-modal variants of standard axioms such as $\T$, $\B$, $\four$, and $\five$, which are shown below.
$$
\T: (A \iimp \xdia A) \land (\xbox A \iimp A) \quad
\four: (\xdia \xdia A \iimp \xdia A) \land (\xbox A \iimp \xbox \xbox A)
$$
$$
\B: (\xdiac A \iimp \xdia A) \land (\xbox A \iimp \xboxc A)
\quad
\five: (\xdiac \xdia A \iimp \xdia A) \land (\xbox A \iimp \xboxc \xbox A)
$$
In the next section, we show that any extension of $\h\ikm(\albet)$ with a set $\axs$ of axioms is sound and complete relative to a specified sub-class of the bi-relational $\albet$-models. For each axiom we extend $\h\ikm(\albet)$ with, we impose a frame condition on our class of bi-relational $\albet$-models. Axioms and related frame conditions are displayed in \fig~\ref{fig:axioms-related-conditions}, and extensions of $\h\ikm(\albet)$ with seriality and $\ipa$ axioms, along with their corresponding models, are defined below.

\begin{figure}[t]
\begin{center}
\bgroup
\def\arraystretch{1.5}
\begin{tabular}{| c | c | c |}
\hline
Axiom & $\xbox A \iimp \xdia A$  & $(\diap{\charx_{1}} \cdots \diap{\charx_{n}} A \iimp \diap{\charx} A) \land (\boxp{\charx} A \iimp \boxp{\charx_{1}} \cdots \boxp{\charx_{n}} A)$ \\
\hline
Condition & $\forall w \exists u (w R_{\charx} u)$ & $\forall w_{0}, \ldots, w_{n} (w_{0} R_{\charx_{1}} w_{1} \land \cdots \land w_{n-1} R_{\charx_{n}} w_{n} \iimp w_{0} R_{\charx} w_{n})$\\
\hline
\end{tabular}
\egroup
\end{center}

\caption{Axioms and their related frame conditions. We note that when $n=0$, the related frame condition is taken to be $w R_{\charx} w$.}
\label{fig:axioms-related-conditions}
\end{figure}

\begin{definition}[Terminology for Extensions] We define the axiomatization $\h\ikm(\albet,\axs)$ to be $\h\ikm(\albet) \cup \axs$, and define the logic $\ikm(\albet,\axs)$ to be the smallest set of formulae from $\lang{\albet}$ closed under substitutions of the axioms and applications of the inference rules. A formula $A$ is defined to be an $\ikm(\albet,\axs)$-theorem, written $\der A$, \ifandonlyif $A \in \ikm(\albet,\axs)$, and a formula $A$ is said to be \emph{derivable} from a set of formulae $\fseti \subseteq \lang{\albet}$, written $\fseti \der A$, \ifandonlyif for some $B_{1}, \ldots, B_{n} \in \fseti$, $\der B_{1} \land \cdots \land B_{n} \iimp A$.

Moreover, we define a \emph{bi-relational $(\albet,\axs)$-model} to be a bi-relational $\albet$-model satisfying each frame condition related to an axiom $A \in \axs$. 
 A formula $A$ is defined to be \emph{globally true} on a bi-relational $(\albet,\axs)$-model $M$, written $M \ent A$, \ifandonlyif $M,u \sat A$ for all worlds $u \in W$ of $M$. A formula $A$ is defined to be \emph{$(\albet,\axs)$-valid}, written $\ent A$, \ifandonlyif $A$ is globally true on every bi-relational $(\albet,\axs)$-model. Last, we say that a set $\fseti$ of formulae \emph{semantically implies} a formula $A$, written $\fseti \ent A$, \ifandonlyif for all bi-relational $(\albet,\axs)$-models $M$ and each $w \in W$ of $M$, if $M, w \sat B$ for each $B \in \fseti$, then $M,w \sat A$.
\end{definition}

\begin{remark} Note that the axiomatization $\h\ikm(\albet) = \h\ikm(\albet,\emptyset)$ and that a bi-relational $(\albet,\emptyset)$-model is a bi-relational $\albet$-model.
\end{remark}

Let us now move on to the next section and prove the soundness and completeness results for our logics.



\section{Soundness and Completeness}\label{sec:sound-complete}

In this section, 
 we show that the $\der$ and $\ent$ relations coincide, that is to say, we show that each intuitionistic grammar logic $\ikma$ is sound and complete. As usual, soundness is straightforward to prove:

\begin{theorem}[Soundness]\label{thm:sound-ikma}
If $\fseti \der A$, then $\fseti \ent A$.
\end{theorem}

\begin{proof} One can prove that if $\der A$, then $\ent A$ by showing that each axiom is valid and each inference rule preserves validity. Then, if we assume that $\fseti \der A$, it follows that for some $B_{1}, \ldots, B_{n} \in \fseti$, $\der B_{1} \land \cdots \land B_{n} \iimp A$, which further implies that $\ent B_{1} \land \cdots \land B_{n} \iimp A$. The last fact permits us to conclude that $\fseti \ent A$.
\qed
\end{proof}

To establish completeness we combine techniques used for establishing the completeness of intuitionistic logic~\cite{Dal04} and intuitionisitc tense logic~\cite{Ewa86}. Our strategy is rather standard and consists of constructing a canonical model where worlds are pairs of the form $(\fseti^{\omega},\fsetii^{\omega})$ with $\fseti^{\omega}$ and $\fsetii^{\omega}$ sets of formulae. If one assumes that $\fseti \not\der A$, then one can show that a pair exists in the canonical model satisfying $\fseti$, but not $A$, thus establishing completeness (see \thm~\ref{thm:complete-ikma} below). We begin by defining two useful notions, viz. the notion of an $\ikma$-consistent set and the notion of an $\ikma$-saturated pair.

\begin{definition}[$\ikma$-Consistent]\label{def:consistent}
We define a pair of sets of formulae $(\fseti,\fsetii)$ to be $\ikma$-\emph{consistent} \ifandonlyif for no finite subsets $\fseti_{0} \subseteq \fseti$ and $\fsetii_{0} \subseteq \fsetii$ we have $\der \bigwedge \fseti_{0} \iimp \bigvee \fsetii_{0}$. 
\end{definition}

\begin{definition}[$\ikma$-Saturated]\label{def:saturated} We define a pair $(\fseti,\fsetii)$ to be $\ikma$-\emph{saturated} \ifandonlyif 
\begin{enumerate}

\item $(\fseti,\fsetii)$ is $\ikma$-consistent;

\item if $\fseti \der A$, then $A \in \fseti$;

\item if $\fseti \der A \lor B$, then $A \in \fseti$ or $B \in \fseti$;

\item $\fseti \cap \fsetii = \emptyset$;

\item $\fseti \cup \fsetii = \lang{\albet}$.

\end{enumerate}
\end{definition}

\begin{lemma}\label{lem:lindenbaum}
Suppose that $(\fseti,\fsetii)$ is $\ikma$-consistent. Then, there exists a saturated pair $(\fseti^{\omega},\fsetii^{\omega})$ such that $\fseti \subseteq \fseti^{\omega}$ and $\fsetii \subseteq \fsetii^{\omega}$.
\end{lemma}

\begin{proof} Let us enumerate all disjunctions from $\lang{\albet}$ where each disjunction occurs infinitely often: $\langle B_{0,i} \lor B_{1,i} \rangle_{i \in \mathbb{N}}$. We set $(\fseti_{0},\fsetii_{0}) := (\fseti,\fsetii)$ and define an infinite sequence of pairs as follows: 
$\fseti_{n+1} := \fseti_{n} \cup \{B_{j,n}\}$ and $\fsetii_{n+1} := \fsetii_{n}$, if $(\fseti_{n} \cup \{B_{j,n}\},\fsetii_{n})$ is $\ikma$-consistent (and if $(\fseti_{n} \cup \{B_{j,n}\},\fsetii_{n})$ is $\ikma$-consistent for both $j = 0$ and $j = 1$, then we set $\fseti_{n+1} := \fseti_{n} \cup \{B_{0,n}\}$), and $\fseti_{n+1} := \fseti_{n}$ and $\fsetii_{n+1} := \fsetii_{n} \cup \{B_{0,i},B_{1,i}\}$ otherwise.

Let $\fseti^{\omega} := \bigcup_{i \in \mathbb{N}} \fseti_{i}$ and $\fsetii^{\omega} := \bigcup_{i \in \mathbb{N}} \fsetii_{i}$. We now argue that $(\fseti^{\omega},\fsetii^{\omega})$ is saturated. It is straightforward to show that for each $n$, $(\fseti_{n},\fsetii_{n})$ is $\ikma$-consistent and that $\fseti_{n} \cap \fsetii_{n} = \emptyset$ from which the saturation properties 1 and 4 can be deduced (see \dfn~\ref{def:saturated} above). We note that saturation properties 3 and 5 follow from the above construction procedure, and 2 follows from 3 since if $\fseti \der A$, then $\fseti \der A \lor A$ (cf.~\cite[\lem~5.3.8]{Dal04}).
\qed
\end{proof}

\begin{definition}[Canonical Model]\label{def:canonical-model} 
 We define the \emph{canonical model} $\canmod(\albet,\axs) := (\candom, \canirel, \{\canmrel \ | \ \charx \in \albet\}, \canv)$ as shown below, and let $w, u \in \candom$ with $w := (\fseti,\fsetii)$ and $u := (\fseti',\fsetii')$.
\begin{itemize}

\item $\candom := \{(\fsetiii,\fsetiv) \ | \ (\fsetiii,\fsetiv) \text{ is saturated.}\}$;

\item $w \canirel u$ \ifandonlyif $\fseti \subseteq \fseti'$; 

\item $w \canmrel u$ \ifandonlyif (i) for all $A \in \lang{\albet}$, if $\xbox A \in \fseti$, then $A \in \fseti'$, and (ii) for all $A \in \lang{\albet}$, if $A \in \fseti'$, then $\xdia A \in \fseti$;

\item $w \in \canv(p)$ \ifandonlyif $p \in \fseti$.

\end{itemize}
\end{definition}

The following two lemmas are proven in an almost identical fashion to \lem~3 and 5 of~\cite{Ewa86}.

\begin{lemma}\label{lem:existence-lemma}
Let $w := (\fseti_{0},\fsetii_{0}) \in \candom$. Then, $\xdia A \in \fseti_{0}$ \ifandonlyif there exists a $u := (\fseti_{1},\fsetii_{1}) \in \candom$ such that $w \canmrel u$ and $A \in \fseti_{1}$.
\end{lemma}



\begin{lemma}\label{lem:necessity-lemma}
Let $w := (\fseti_{0},\fsetii_{0}) \in \candom$. Then, $\xbox A \in \fseti_{0}$ \ifandonlyif for each $u := (\fseti_{1},\fsetii_{1})$ and $v := (\fseti_{2},\fsetii_{2})$ in $\candom$, if $w \canirel u$ and $u \canmrel v$, then $A \in \fseti_{2}$.
\end{lemma}



\begin{lemma}\label{lem:containment-lemma}
The canonical model $\canmod(\albet,\axs)$ is a bi-relational $(\albet,\axs)$-model.
\end{lemma}


\begin{proof} It is straightforward to show that $\canmod$ is a bi-relational $(\albet,\axs)$-model. The proof that $\canmod$ satisfies properties (F1)--(F3) uses axioms A8, A9, and A7, respectively (cf.~\cite{Ewa86}), and the fact that the valuation function $\canv$ is monotonic follows from its definition and the definition of $\canirel$. Below, we show that $\canmod$ satisfies each frame property associated with an axiom from $\axs$.
\begin{itemize}
\item[$\D$] We show that if the seriality axiom $\xbox A \iimp \xdia A$ is included in our axiomatization, then $R_{\charx}$ is serial. Let $w := (\fseti,\fsetii) \in \candom$ and observe that the formula $\xbox(p \iimp p) \in \fseti$ since if it were in $\fsetii$, $w$ would not be $\ikma$-consistent (and hence, not saturated). Therefore, by applying the seriality axiom $\D$, we may conclude that $\xdia (p \iimp p) \in \fseti$, from which it follows that there exists a $u := (\fsetiii,\fsetiv) \in \candom$ such that $w \canmrel u$ by \lem~\ref{lem:existence-lemma}.

\item[$\ipa$] We show that if $(\diap{\charx_{1}} \cdots \diap{\charx_{n}} A \iimp \diap{\charx} A) \land (\boxp{\charx} A \iimp \boxp{\charx_{1}} \cdots \boxp{\charx_{n}} A)$ is included in our axiomatization, then for any $w_{0}, \ldots, w_{n} \in \candom$, if $w_{i} R_{\charx_{i+1}}^{C} w_{i+1}$ for each $i \in \{0, \ldots, n-1\}$, then $w_{0} R_{\charx}^{C} w_{n}$. Let $w_{0}, \ldots, w_{n}$ be arbitrary worlds in $\candom$ and suppose that $w_{i} R_{\charx_{i+1}}^{C} w_{i+1}$ for each $i \in \{0, \ldots, n-1\}$. We aim to show that $w_{0} R_{\charx}^{C} w_{n}$, where $w_{0} := (\fseti_{0},\fsetii_{0})$ and $w_{n} := (\fseti_{n}, \fsetii_{n})$. First, assume that $\xbox A \in \fseti_{0}$. Then, by the above axiom, $\boxp{\charx_{1}} \cdots \boxp{\charx_{n}} A \in \fseti_{0}$, and by our assumption and the definition of the $\canmrel$ relation, $A \in \fseti_{n}$. Second, assume that $A \in \fseti_{n}$. Then, by our assumption and the definition of the $\canmrel$ relation, $\diap{\charx_{1}} \cdots \diap{\charx_{n}} A \in \fseti_{0}$, so by the above axiom, $\xdia A \in \fseti_{0}$. Therefore, $w_{0} R_{\charx}^{C} w_{n}$.
\end{itemize}
\qed
\end{proof}

\begin{lemma}[Truth Lemma]\label{lem:truth-lemma}
Let $w := (\fseti,\fsetii)$ be saturated. Then, we have $\canmod(\albet,\axs), w \sat A$ \ifandonlyif $A \in \fseti$.
\end{lemma}


\begin{proof} We prove the result by induction on the complexity of $A$ and argue the $\lor$, $\iimp$, $\xdia$, and $\xbox$ cases since the other cases are simple.

$B \lor C$. $\canmod(\albet,\axs), w \sat B \lor C$ \ifandonlyif $\canmod(\albet,\axs), w \sat B$ or $\canmod(\albet,\axs), w \sat C$ \ifandonlyif $B \in \fseti$ or $C \in \fseti$ \ifandonlyif $B \lor C \in \fseti$. We note that the second `\textit{iff}' follows from IH an the third follows from the fact that $w$ is saturated (see \dfn~\ref{def:saturated}).

$B \iimp C$. The right-to-left direction is straightforward, so we show the left-to-right direction by contraposition. 
 Suppose that $B \iimp C \not\in \fseti$. It follows that $\fseti \cup \{B\} \not\der C$, implying that the pair $(\fseti \cup \{B\}; \{C\})$ is $\ikma$-consistent, and so, we may extend it to a saturated pair $u := (\fsetiii,\fsetiv)$. Observe that $\fseti \subseteq \fsetiii$, $B \in \fsetiii$, and $C \not\in \fsetiii$. By the definition of $\canirel$ and IH, it follows that $u \in \candom$ with $w \canirel u$, $\canmod(\albet,\axs), u \sat B$, and $\canmod(\albet,\axs), u \not\sat C$, entailing that $\canmod(\albet,\axs), w \not\sat B \iimp C$.

$\xdia B$. The left-to-right direction is straightforward, so we show the right-to-left direction. Suppose that $\xdia B \in \fseti$. Then, by \lem~\ref{lem:existence-lemma} we know that there exists a $u := (\fsetiii,\fsetiv) \in \candom$ such that $w \canmrel u$ and $B \in \fsetiii$. Therefore, $\canmod(\albet,\axs), u \sat B$ by IH, implying that $\canmod(\albet,\axs), w \sat \xdia B$.

$\xbox B$. Follows from \lem~\ref{lem:necessity-lemma} and IH.
\qed
\end{proof}

\begin{theorem}[Completeness]\label{thm:complete-ikma}
If $\fseti \ent A$, then $\fseti \der A$.
\end{theorem}

\begin{proof} Suppose $\fseti \not\der A$. Then, $(\fseti,\{A\})$ is $\ikma$-consistent and can be extended to a saturated pair $w := (\fseti^{\omega},\fsetii^{\omega})$. By \lem~\ref{lem:truth-lemma}, $\canmod, w \sat B$ for each $B \in \fseti^{\omega}$, but  $\canmod, w \not\sat C$ for each $C \in \fsetii^{\omega}$. Hence, $\fseti \not\ent A$.
\qed
\end{proof}



\section{Conclusion}\label{sec:conclusion}

This paper provided sound and complete axiomatizations for intuitionistic grammar logics. We defined a base intuitionistic grammar logic $\ikm(\albet)$, for each alphabet $\albet$, and provided axiomatizations for extensions of $\ikm(\albet)$ with combinations of seriality axioms and intuitionistic path axioms. In future work, we aim to provide nested sequent systems in the style of~\cite{Str13} for the logics discussed here by making use of the structural refinement methodology of~\cite{Lyo21thesis}. The goal will be to identify decidable fragments of intuitionistic grammar logics via proof-search. Moreover, due to the connection between modal logics and description logics, it could be worthwhile to investigate the use of intuitionistic grammar logics (or close variants thereof) in knowledge representation.

\bibliographystyle{splncs04}
\bibliography{bibliography}

\end{document}